\documentclass[a4paper, 12pt, reqno]{amsart}

\usepackage{pstricks, pst-node,pst-plot}
\usepackage{amssymb, a4, mathrsfs, color}

\newtheorem{thm}{Theorem}[section]
\newtheorem{prop}[thm]{Proposition}
\newtheorem{lem}[thm]{Lemma}

\theoremstyle{definition}
\newtheorem{defn}[thm]{Definition}

\theoremstyle{definition}

\numberwithin{thm}{section} 
\numberwithin{equation}{section}
\numberwithin{figure}{section}

\newcommand{\ix}{\mathcal{I}_X}
\newcommand{\bd}{{\bf d}}
\newcommand{\br}{{\bf r}}
\newcommand{\bc}{{\bf C}}
\newcommand{\bs}{{\bf S}}

\newcommand{\thec}{\Theta}
\newcommand{\phis}{\Phi}

\newcommand{\dom}{\textup{dom\,}}

\newcommand{\ptx}{\mathcal{PT}_X}
\newcommand{\R}{\mathcal{R}}

\newcommand{\RT}{\widetilde{\R}}
\newcommand{\RTE}{\widetilde{\R}_E}
\newcommand{\LL}{\mathcal{L}}

\newcommand{\LT}{\widetilde{\LL}}
\newcommand{\LTE}{\LT_E}

\newcommand{\szgg}{\textup{Sz}(G)}
\newcommand{\pb}{\overline{\psi}}

\begin{document}

\title[Extending the ESN Theorem]{Extending the Ehresmann-Schein-Nambooripad Theorem}
\author{Christopher Hollings}
\address{Centro de \'{A}lgebra da Universidade de Lisboa\\ Av.\ Prof.\ Gama Pinto 2\\ 1649-003 Lisboa\\ Portugal}
\email{cdh500@cii.fc.ul.pt}
\date{\today}
\subjclass[2000]{20\,M\,18,\ 20\,L\,05,\ 20\,M\,15}
\thanks{This work was completed as part of Project POCTI/0143/2007 of CAUL, financed by FCT and FEDER, and also as part of FCT post-doctoral research grant SFRH/BPD/34698/2007.  Thanks must go to both Mark Lawson and Victoria Gould for a number of useful comments.}
\begin{abstract}
We extend the `$\vee$-premorphisms' part of the
Ehres\-mann-Schein-Nambooripad Theorem to the case of two-sided
restriction semigroups and inductive categories, following on from
a result of Lawson (1991) for the `morphisms' part. However, it is
so-called `$\wedge$-premorphisms' which have proved useful in
recent years in the study of partial actions.  We therefore obtain
an Ehresmann-Schein-Nambooripad-type theorem for
(ordered) $\wedge$-premorphisms in the case of two-sided restriction
semigroups and inductive categories.  As a corollary, we obtain
such a theorem in the inverse case.
\end{abstract}
\maketitle

\section{Introduction}
\label{sec:intro}

In their study of $E$-unitary covers for inverse semigroups,
McAlister and Reilly \cite{mr1977} made the following definitions:

\begin{defn}\cite[Definition~3.4]{mr1977}
\label{defn:join-i}
Let $S$ and $T$ be inverse semigroups.  A
\emph{$(\vee,i)$-premorphism} is a function $\theta:S\rightarrow
T$ such that
\begin{enumerate}
    \item[($\vee 1$)] $(st)\theta\leq (s\theta)(t\theta)$.
\end{enumerate}
\end{defn}

\begin{defn}\cite[Definition~4.1]{mr1977}
\label{defn:wedge-i} Let $S$ and $T$ be inverse semigroups.  A
\emph{$(\wedge,i)$-premorphism} is a function $\theta:S\rightarrow
T$ such that\footnote{We place a $'$ on condition $(\wedge 2)'$,
as we will shortly replace this condition by a weaker one and we
wish to reserve the label $(\wedge 2)$ for that.}
\begin{enumerate}
    \item[$(\wedge 1)$] $(s\theta)(t\theta)\leq (st)\theta$;
    \item[$(\wedge 2)'$] $(s\theta)^{-1}=s^{-1}\theta$.
\end{enumerate}
\end{defn}

(We note that in \cite{mr1977}, a $(\vee,i)$-premorphism was
termed a \emph{$v$-prehomomorphism}, whilst in
\cite[p.~80]{lawson1998}, it is termed simply a
\emph{prehomomorphism}.  In \cite{mr1977}, a
$(\wedge,i)$-premorphism was called a
\emph{$\wedge$-prehomomorphism}; in \cite[p.~80]{lawson1998}, it
is called a \emph{dual prehomomorphism}.)

These functions were central to McAlister and Reilly's
constructions: see their Theorems~3.9 and~4.5 \cite{mr1977}.

In the present paper, we will study generalisations of the
functions of Definitions~\ref{defn:join-i} and~\ref{defn:wedge-i}
in the case of so-called \emph{two-sided restriction semigroups}
(until recently, termed \emph{weakly $E$-ample semigroups}). These
are an extremely natural class of semigroups which generalise
inverse semigroups and which arise from partial transformation
monoids in a manner analogous to the way in which inverse
semigroups arise from symmetric inverse monoids.  The concrete
description of such a semigroup is as follows.  Let $\ptx$ denote
the collection of all partial mappings of a set $X$, i.e., all
mappings $A\rightarrow B$, where $A,B\subseteq X$.  We compose
elements of $\ptx$ (from left to right) according to the usual rule for composition of
partial mappings, namely, that employed in the symmetric inverse
monoid $\ix$. Under this composition, $\ptx$ clearly forms a
monoid, which we term the \emph{partial transformation monoid of $X$}.  Similarly, $\ptx^*$, the \emph{dual partial transformation monoid of $X$}, is the collection of all partial mappings of $X$ with composition performed \emph{from right to left}.  
We denote by $I_A$ the partial identity mapping on a
subset $A\subseteq X$; the collection $E_X$ of all such partial
identities forms a subsemilattice of both $\ptx$ and $\ptx^*$.  We now consider the unary operation on partial transformations which is given by
$\alpha\mapsto I_{\dom\alpha}.$  
In $\ptx$, we denote this operation by $^+$; in $\ptx^*$, we denote it by $^*$.
Let $S$ be a semigroup.  We call $S$ a \emph{two-sided restriction semigroup} if
\begin{enumerate}
	\item $S$ is isomorphic to a subsemigroup of some $\ptx$ that is closed under $^+$ (via an isomorphism $\phi$);
	\item $S$ is isomorphic to a subsemigroup of some $\mathcal{PT}_Y^*$ that is closed under $^*$ (via an isomorphism $\psi$);
	\item the semilattices $\{(s\phi)^+:s\in S\}$ and $\{(s\psi)^*:s\in S\}$ are isomorphic.
\end{enumerate}
Such semigroups have appeared in a range of contexts
(see \cite{gh}) and have an alternative, abstract characterisation
which will be used throughout this paper (see
Section~\ref{sec:rest}).

In extending the functions of Definitions~\ref{defn:join-i}
and~\ref{defn:wedge-i} to the case of two-sided restriction
semigroups, our particular interest is in obtaining a
generalisation of the celebrated
\emph{Ehresmann-Schein-Nambooripad Theorem} (hereafter, \emph{ESN
Theorem}).  This theorem establishes a fundamental connection
between inverse semigroups and inductive groupoids.  The formal
statement of the ESN Theorem (as it appears in \cite{lawson1998},
for which book it provides the main focus) is as follows:
\begin{thm}\cite[Theorem~4.1.8]{lawson1998}
\label{esn} The category of inverse semigroups and
$(\vee,i)$-premorphisms is isomorphic to the category of inductive
groupoids and ordered functors; the category of inverse semigroups
and morphisms is isomorphic to the category of inductive groupoids
and inductive functors.
\end{thm}
(An \emph{ordered functor} is simply an order-preserving functor,
whilst a \emph{inductive functor} is an ordered functor which also
preserves the `meet' operation in an inductive groupoid.)

The `morphisms' part of this result has already been extended to
the case of two-sided restriction semigroups by Lawson
\cite{lawson1991}; in Lawson's result, inductive groupoids are
replaced by inductive categories and morphisms by
(2,1,1)-morphisms. We will complete the generalisation by
considering the `$(\vee,i)$-pre\-morph\-isms' part. Furthermore, we
will obtain a version of this theorem for
$(\wedge,i)$-premorphisms, via the more general case of two-sided restriction semigroups; it is this type of premorphism
which has proved most useful in recent years in the study of
partial actions.  We note that arbitrary $(\wedge,i)$-premorphisms
do not compose to give another $(\wedge,i)$-premorphism. This
problem is solved (both in the inverse case, and in the theory we
will develop for restriction semigroups) if we also insist that
the functions be order-preserving.

The structure of the paper is as follows.  We begin with
Section~\ref{sec:rest}, in which we record some results on
two-sided restriction semigroups which will be of use in later
sections. Further preliminary definitions and results follow in Section~\ref{sec:prelims}: in the first part of the section, we introduce various category-theoretic notions, including that of an inductive category; in the second part, we define the \emph{Szendrei expansion} of an inductive groupoid --- a piece of algebraic machinery which will be needed briefly in Section~\ref{sec:inv}.  In Section~\ref{sec:join} we define the notion of
\emph{$(\vee,r)$-premorphisms} for two-sided restriction
semigroups and prove a new version (Theorem~\ref{thm:wEa-join}) of the ESN Theorem with these
functions as the arrows of a category of two-sided
restriction semigroups. We show that the existing result in the
inverse case follows from ours as a corollary.  In
Section~\ref{sec:meet}, we move onto the `$\wedge$-premorphisms'
part of the paper.  By adapting concepts encountered in the study
of partial actions, we prove yet another version (Theorem~\ref{thm:wEajoin-indcat}) of the ESN
Theorem, this time for two-sided restriction semigroups and
\emph{ordered $(\wedge,r)$-premorphisms}.  In fact, we will need to prove
two versions of this theorem: it will turn out that the most naive
notion of `$(\wedge,r)$-premorphism', i.e., that obtained by replacing condition $(\wedge 2)'$ in Definition~\ref{defn:wedge-i} by a condition involving $^+$ and $^*$, will not suffice for our
purposes and we will therefore also need the notion of a
\emph{strong} $(\wedge,r)$-premorphism (see Theorem~\ref{thm:wEajoin-indcat-strong}).  A version of the ESN
Theorem for inverse semigroups and ordered $(\wedge,i)$-premorphisms (Theorem~~\ref{thm:esn-new}) will
follow in Section~\ref{sec:inv}, as a corollary to the results of
Section~\ref{sec:meet}, thanks to our introduction of strong $(\wedge,r)$-premorphisms.

\section{Restriction semigroups}
\label{sec:rest}

In this section, we summarise the pertinent details of the theory
of two-sided restriction semigroups.  The material of this section
appears in a range of published sources (see, for example,
\cite{fountain1977b,fountain1979,gg2003,js2001,lawson1991}).
However, there are only two places where many of the relevant
definitions and results have been collated into a single resource:
the notes \cite{gould} and Chapter~2 of the author's Ph.D.\
thesis~\cite{thesis}.  The reader is referred to these sources for
further details and for more extensive references.  For a more
easily accessible published source, see Section~1 of \cite{h2007},
written for monoids.

Notice from our comments in the Introduction that we can easily
define one-sided versions of these semigroups: \emph{left
restriction semigroups} (subsemigroups of $\ptx$ closed under $^+$) and \emph{right
restriction semigroups} (subsemigroups of $\mathcal{PT}_Y^*$ closed under $^*$). However, in this
paper, we will only be interested in the two-sided version.  From
here on, we therefore drop the qualifier `two-sided'; henceforth,
whenever we use the term `restriction semigroup', it can be taken
to mean the two-sided version.  We also note here that
left/right/two-sided restriction semigroups have appeared in a
range of contexts under a number of different names: see \cite{gh}
for further references.  The term `restriction semigroup' is a
recent attempt to harmonise terminology and originates with
\cite{cm}.

In the Introduction, we saw the concrete characterisation of
restriction semigroups as subsemigroups of partial transformation
monoids.  We now give an abstract description.  Let
$S$ be a semigroup and suppose that $E\subseteq E(S)$ is a
subsemilattice of $S$.  We define the (equivalence) relations
$\RT_E$ and $\LT_E$ on $S$, with respect to $E$, by the rules that
\begin{gather*}
a\,\RTE\,b\Longleftrightarrow\forall e\in E\ [ea=a\Leftrightarrow eb=b]; \\
a\,\LTE\,b\Longleftrightarrow\forall e\in E\ [ae=a\Leftrightarrow
be=b],
\end{gather*}
for $a,b\in S$.  Thus, two elements $a,b$ are $\RT_E$-related if,
and only if, they have the same left identities in $E$. Similarly,
$a\,\LTE\,b$ if, and only if, $a$ and $b$ have the same right
identities in $E$.  The abstract definition of a restriction
semigroup runs as follows:

\begin{defn}
\label{defn:wlEa} A semigroup $S$ with subsemilattice $E\subseteq
E(S)$ is \emph{(two-sided) restriction semigroup (with respect to
$E$)} if
\begin{enumerate}
    \item every element $a$ is both $\RTE$- and $\LTE$-related to an idempotent in $E$, denoted $a^+$ and $a^*$, respectively;
    \item $\RTE$ is a left congruence, whilst $\LTE$ is a right congruence;
    \item for all $a\in S$ and all $e\in E$, $ae=(ae)^+a$ and $ea=a(ea)^*$.
\end{enumerate}
\end{defn}
Thus, in a restriction semigroup $S$,
$$a\,\RTE\,b\Leftrightarrow a^+=b^+\qquad\textup{and}\qquad a\,\LTE\,\Leftrightarrow a^*=b^*.$$
The idempotents $a^+$ and $a^*$ are left and right identities for
$a$, respectively.  We note that $a^+$ and $a^*$ are necessarily
unique.  It is also clear that if $e\in E$, then $e^+=e=e^*$.  

If we were to consider only the parts of
Definition~\ref{defn:wlEa} which relate to $\RTE$ and $^+$
(respectively, $\LTE$ and $^*$), then we would have a \emph{left}
(respectively, \emph{right}) restriction semigroup.

Using \cite[Theorem~6.2]{gould}, it is possible to connect the
concrete and abstract approaches to restriction semigroups by
showing that left (right) restriction semigroups are, up to
isomorphism, \emph{precisely} (2,1)-subalgebras of (dual) partial
transformation monoids.  However, two-sided restriction semigroups
cannot be regarded as (2,1,1)-subalgebras of partial
transformation monoids.

Note that $\R\subseteq\RT_E$ and $\LL\subseteq\LTE$, for any $E$.  It is easy to see that, in a regular semigroup, $\R=\RT_{E(S)}$ and
$\LL=\LT_{E(S)}$.  It follows that restriction semigroups
generalise inverse semigroups, since every inverse semigroup is a
restriction semigroup with $a^+=aa^{-1}$ and $a^*=a^{-1}a$.
Restriction semigroups also generalise the ample (formerly,
type-A) semigroups of Fountain~\cite{fountain1977b,fountain1979}.

We note a pair of useful identities which follow easily from
condition (2) of Definition~\ref{defn:wlEa}:
\begin{lem}\cite[Proposition~1.6(2)]{fountain1979}
\label{lem:+} Let $S$ be a restriction semigroup, for some
$E\subseteq E(S)$, and let $s,t\in S$.  Then $(st)^+=(st^+)^+$ and
$(st)^*=(s^*t)^*$.
\end{lem}

Both $\ptx$ and $\ptx^*$ possess an obvious
natural partial order (i.e., a partial order which is compatible
with multiplication and which restricts to the usual partial order
on idempotents), defined by
\begin{equation}
\label{eq:ordering}
\alpha\leq\beta\Longleftrightarrow\alpha=\beta|_{\dom\alpha}.
\end{equation}
In the abstract characterisation of a restriction semigroup $S$,
the ordering of \eqref{eq:ordering} becomes the following natural
partial order
\begin{equation}
\label{eq:npo} 
a\leq b\Longleftrightarrow a=eb\Longleftrightarrow a=bf,
\end{equation}
for some idempotents $e,f\in E$.  Equivalently,
\begin{equation}
\label{eq:npo2} 
a\leq b\Longleftrightarrow a=a^+b\Longleftrightarrow a=ba^*.
\end{equation}
This equivalence is justified (for the `$^+$' part) in \cite[\S
1]{h2007}.

\section{Further preliminaries}
\label{sec:prelims}

In this section, we describe the relevant existing results connecting restriction semigroups and inductive categories, and introduce some algebraic machinery for later use.

\subsection{Categories}
\label{sub:cat}
We begin by giving an explicit definition of an arbitrary
category.  The results quoted in this subsection
will take care of the `objects' parts of our main results: the upcoming ESN-type Theorems~\ref{thm:wEa-join}, \ref{thm:wEajoin-indcat}, \ref{thm:wEajoin-indcat-strong} and~\ref{thm:esn-new}.

Let $C$ be a class and let $\cdot$ be a partial binary operation
on $C$. For $x,y\in C$, we will write `$\label{cat-prod}\exists
x\cdot y$' to mean `the product $x\cdot y$ is defined'.  Whenever
we write `$\exists (x\cdot y)\cdot z$', it will be understood that
we mean $\exists x\cdot y$ \emph{and} $\exists (x\cdot y)\cdot z$.
An element $e\in C$ is \emph{idempotent} if $\exists e\cdot e$ and
$e\cdot e=e$. The \emph{identities} (or \emph{objects}) of $C$ are
those idempotents $e$ which satisfy: $[\exists e\cdot x\Rightarrow
e\cdot x=x]$ and $[\exists x\cdot e\Rightarrow x\cdot e=x]$.  We
denote the subset of identities of $C$ by $C_o$ (`\textit{o}' for
`objects').
\begin{defn}
\label{defn:cat} Let $C$ be a class and let $\cdot$ be a partial
binary operation on $C$.  The pair $(C,\cdot)$ is a
\emph{category} if the following conditions hold:
\begin{enumerate}
    \item[(Ca1)] $\exists x\cdot(y\cdot z)\Longleftrightarrow \exists
(x\cdot y)\cdot z$, in which case $x\cdot(y\cdot z)=(x\cdot
y)\cdot z$;
    \item[(Ca2)] $\exists x\cdot(y\cdot z)\Longleftrightarrow \exists
x\cdot y$ and $\exists y\cdot z$;
    \item[(Ca3)] for each $x\in C$, there exist unique identities
$\bd(x),\br(x)\in C_o$ such that $\exists \bd(x)\cdot x$ and
$\exists x\cdot\br(x)$.
\end{enumerate}
The identity ${\bf d}(x)$ is called the \emph{domain} of $x$ and
${\bf r}(x)$ is the \emph{range} of $x$.  If $C$ is simply a set,
then we call $(C,\cdot)$ a \emph{small} category.  A \emph{groupoid} is a small category in which the following additional condition holds:
\begin{enumerate}
    \item[(G)] for each $x\in G$, there is an $x^{-1}\in G$ such that $\exists x\cdot x^{-1}$ and $\exists x^{-1}\cdot x$, with $x\cdot x^{-1}=\bd(x)$ and $x^{-1}\cdot x=\br(x)$.
\end{enumerate}
\end{defn}
This is essentially the definition of \cite[p.~78]{lawson1998},
but with domain and range switched, since we will be composing
functions from left to right.  We note that if $(C,\cdot)$ is a
category and $x,y\in C$, then
$$\exists x\cdot y\Longleftrightarrow \br(x)=\bd(y).$$
Note further that if $(C,\cdot)$ has precisely one object $e$,
then all products are necessarily defined and it follows that
$(C,\cdot)$ is, in fact, a monoid with identity $e$.  Thus, a
category may be regarded as a generalisation of a monoid.

Before proceeding further, we make the important observation that
in the remainder of this paper, the word `category' will be used
in two slightly different, though equivalent, senses.  First, we
will have `categories' as generalised monoids, in the sense of
Definition~\ref{defn:cat}; second, we will have `categories' in
the more traditional `objects' and `morphisms' sense (as in, for
example, \cite[Definition 1.1]{jacobson1980}).  All categories
considered in the first sense will be small categories.  
Thus, for example, the `inductive categories' of
Theorem~\ref{thm:wEa-join} are (small) categories in the sense of
Definition~\ref{defn:cat}, but the category \emph{of} inductive
categories (and that of restriction semigroups) is regarded as a
category in the traditional `objects' and `morphisms' sense.

We now introduce an ordering on a category $C$:
\begin{defn}
\label{defn:ord-cat} Let $(C,\cdot)$ be a category (in the sense
of Definition~\ref{defn:cat}) and let $C$ be partially ordered by
$\leq$.  The triple $(C,\cdot,\leq)$ is an \emph{ordered} category
if the following conditions hold:
\begin{enumerate}
    \item[(Or1)] $a\leq c$, $b\leq d$, $\exists a\cdot b$ and
$\exists c\cdot d\Longrightarrow a\cdot b\leq c\cdot d$;
    \item[(Or2)] $a\leq b\Longrightarrow {\bf r}(a)\leq {\bf r}(b)$ and ${\bf d}(a)\leq {\bf d}(b)$;
    \item[(Or3)]    \begin{enumerate}
            \item[(i)] $\label{Or3}f\in C_o,a\in C,f\leq {\bf r}(a)\Longrightarrow$
there exists a unique element, denoted $a|f$, such that $a|f\leq
a$ and ${\bf r}(a|f)=f$;
            \item[(ii)] $f\in C_o,a\in C,f\leq {\bf d}(a)\Longrightarrow$
there exists a unique element, denoted $f|a$, such that $f|a\leq
a$ and ${\bf d}(f|a)=f$.
            \end{enumerate}
\end{enumerate}
An \emph{ordered groupoid} is a small ordered category in which condition (G) holds.
\end{defn}
The element $a|f$ of condition (Or3)(i) is called the
\emph{corestriction} (of $a$ to $f$), whilst the element $f|a$ of
condition (Or3)(ii) is called the \emph{restriction} (of $f$ to
$a$).  We note that, for $e,f\in C_o$, there is no ambiguity in
the notation `$e|f$': the restriction $e|f$ and the corestriction
$\label{coincide} e|f$ coincide whenever both are defined.  Note
further that in this case we have $e=f=e|f$, by definition of
restrictions and corestrictions.  
It is clear that any function
which respects domains, ranges and ordering (for example, an
ordered functor) will also respect restrictions and
corestrictions.

We record the following properties of the restriction and corestriction in an ordered category:

\begin{lem}
\label{lem:rest-ord} Let $(C,\cdot,\leq)$ be an ordered category.
Let $a\in C$ and $e,f\in C_o$ with $f\leq e\leq\br(a)$.  Then
$(a|e)|f=a|f$.  Consequently, $a|f\leq a|e$.

Dually, if $a\in C$ and $e,f\in C_o$ with $f\leq e\leq\bd(a)$,
then $f|(e|a)=f|a$, hence $f|a\leq e|a$.
\end{lem}

(Our Lemma~\ref{lem:rest-ord} is essentially Lemma~4.4 of \cite{lawson1991}.)

Note that in an inductive groupoid $G$, corestrictions may be defined in terms of restrictions: for $a\in G$ and $f\in G_o$ with $f\leq\br(a)$, the
corestriction $a|f$ is given by $a|f=(f|a^{-1})^{-1}$ \cite[p.~178]{gilbert2005}.

\begin{lem}
\label{lem:a=dab} Let $(C,\cdot,\leq)$ be an ordered category and
let $a,b\in C$.  If $a\leq b$, then $a=\bd(a)|b=a|\br(b)$.  In
particular, $a=\bd(a)|a=a|\br(a)$.
\end{lem}

(This result appears in \cite{armstrong1984} for inductive cancellative categories; the proof carries over the present case without modification.)

In an ordered category $(C,\cdot,\leq)$, if the greatest lower
bound of $e,f\in C_o$ exists, then we denote it by $e\wedge f$.

\begin{defn}
An \emph{inductive} category $(C,\cdot,\leq)$ is an ordered
category in which the following additional condition holds:
\begin{enumerate}
    \item[(I)] $e,f\in C_o\Longrightarrow e\wedge f$ exists in $C_o$.
\end{enumerate}
An \emph{inductive groupoid} is a small inductive category in which condition (G) holds.
\end{defn}

The foregoing sequence of definitions has been building to the
following result, which appears in \cite[\S 5]{lawson1991}:
\begin{thm}
\label{thm:wEa-ic} Let $S$ be a restriction semigroup with respect to some semilattice $E$, and let $S$ have natural
partial order $\leq$.  If we define the \emph{restricted product}
$\cdot$ in $S$ by
\begin{equation*}
\label{eq:restricted} a\cdot b=   \begin{cases}
                            ab                              &\text{if }a^*=b^+; \\
                            \text{undefined}    &\text{otherwise}, \\
                        \end{cases}
\end{equation*}
then $(S,\cdot,\leq)$ is an inductive category with $S_o=E$, ${\bf
d}(x)=x^+$ and ${\bf r}(x)=x^*$.
\end{thm}

The restriction $f|a$ in the inductive category $(S,\cdot,\leq)$
is simply the product $fa$ in the restriction semigroup $S$, since
$fa\leq a$ and $(fa)^+=(fa^+)^+=f$, by Lemma~\ref{lem:+}, as
$f\leq a^+$. Similarly, $a|f=af$ and $e\wedge f=ef$.

We now define the \emph{pseudoproduct} $\otimes$ in an inductive
category $(C,\cdot,\leq)$ by\footnote{In the interests of reducing
the number of brackets, expressions such as $a|\br(a)\wedge\bd(b)$
will be understood to mean $a|(\br(a)\wedge\bd(b))$; of course,
the alternative, $(a|\br(a))\wedge \bd(b)$, makes no sense if $a$
is not an identity.}
\begin{equation*}
\label{eq:pseudo-cat} a\otimes b=\left[a|{\bf r}(a)\wedge {\bf
d}(b)\right]\cdot\left[{\bf r}(a)\wedge {\bf d}(b)|b\right].
\end{equation*}
The pseudoproduct is everywhere-defined in $C$ (thanks to (I)) and
coincides with the product $\cdot$ in $C$ whenever $\cdot$ is
defined.  To see this, recall that if $\exists a\cdot b$, then
$\br(a)=\bd(b)$, so that $a\otimes
b=\left[a|\br(a)\right]\cdot\left[\bd(b)|b\right]$.  Then
$a\otimes b=a\cdot b$, by Lemma~\ref{lem:a=dab}.

We record the following properties of the pseudoproduct for later use:

\begin{lem}
\label{lem:apseudoe} 
Let $(C,\cdot,\leq)$ be an inductive category and let $a\in C$, $e\in C_o$.  Then
$$e\otimes a=e\wedge {\bf d}(a)|a\quad\text{and}\quad a\otimes e=a|{\bf r}(a)\wedge e.$$
\end{lem}
(This result appears in \cite{armstrong1984} for inductive cancellative categories; the proof carries over the present case without modification.)

\begin{thm}\cite[\S 5]{lawson1991}
\label{thm:ic->wEas} If $(C,\cdot,\leq)$ is an inductive category,
then $(C,\otimes)$ is a restriction semigroup with respect to
$C_o$.
\end{thm}

Let $S$ be a restriction semigroup.  We will denote the inductive
category associated to $S$ by ${\bf C}(S)$.  Similarly, if $C$ is
an inductive category, then we will denote its associated
restriction semigroup by ${\bf S}(C)$.  The following result is
implicit in \cite[\S 5]{lawson1991}:

\begin{thm}
\label{thm:scs-csc} Let $S$ be a restriction semigroup and $C$ be
an inductive category.  Then $\bs(\bc(S))=S$ and $\bc(\bs(C))=C$.
\end{thm}

Theorems~\ref{thm:wEa-ic}, \ref{thm:ic->wEas} and~\ref{thm:scs-csc} prove the equivalence of the objects of the categories in the upcoming Theorems~\ref{thm:wEa-join}, \ref{thm:wEajoin-indcat} and \ref{thm:wEajoin-indcat-strong}.  It only remains for us to deal with the arrows.

We conclude this subsection with a property of ordered categories which will be immensely useful later in the paper.

\begin{lem}
\label{lem:so-useful} 
Let $(C,\cdot,\leq)$ be an ordered category
and let $a,b\in C$.  If $a,b\leq c$, for some $c\in C$, and either
$\bd(a)=\bd(b)$ or $\br(a)=\br(b)$, then $a=b$.
\end{lem}

(This result appears in \cite{armstrong1984} for inductive cancellative categories; the proof carries over the present case without modification.)

\subsection{The Szendrei expansion of an inductive groupoid}
\label{sec:sz}

We take this opportunity to introduce an algebraic tool which we will require briefly in Section~\ref{sec:inv}: the \emph{Szendrei expansion} of an inductive groupoid, as defined by Gilbert~\cite{gilbert2005}.

The concept of a `semigroup expansion' was first introduced by
Birget and Rhodes in \cite{br1984}, and is simply a special type
of functor from one category of semigroups to another.  Amongst
the various expansions in the literature are a number of different
versions of the so-called \emph{Szendrei expansion}.  The original
Szendrei expansion was introduced in \cite{szendrei1989} and was
there applied to groups; it was subsequently extended to monoids
in \cite{fg1990}.

Gilbert~\cite{gilbert2005} has studied the partial actions of
inductive groupoids as a way of informing the study of partial
actions of inverse semigroups.  In the course of this work, he has defined a Szendrei
expansion for inductive groupoids.  Let $G$ be an inductive
groupoid.  For each identity $e\in G_o$, we define the set
$\text{star}_e(G)$ by
\begin{equation*}
\text{star}_e(G)=\{g\in G:\bd(g)=e\}.
\end{equation*}
Let $\mathscr{F}_e(\text{star}_e(G))$ be the collection of all
finite subsets of $\text{star}_e(G)$ which contain $e$, and put
$$\mathscr{F}_*(G)=\bigcup_{e\in G_o}\mathscr{F}_e(\text{star}_e(G)).$$
We now make the following definition:
\begin{defn}\cite[p.~179]{gilbert2005}
\label{defn:grpd-exp} Let $G$ be an inductive groupoid.  The
\emph{Szendrei expansion}\footnote{Gilbert refers to this as the
\emph{Birget-Rhodes expansion} but we adopt the term
\emph{Szendrei expansion} for consistency with previous
definitions.} of $G$ is the set
$$\szgg=\{(U,u)\in\mathscr{F}_*(G)\times G:u\in U\},$$
together with the operation
$$(U,u)(V,v)=
\begin{cases}
(U,uv)                      &\text{if }\br(u)=\bd(v)\text{ and }U=uV; \\
\text{undefined}    &\text{otherwise}.
\end{cases}$$
\end{defn}

We note that $\szgg$ has identities
$$\szgg_o=\{(E,e)\in\szgg:e\in G_o\}.$$
We note also that Gilbert's definition applies, more generally, to
\emph{ordered} group\-oids, but we will only concern ourselves
with the inductive case.

\begin{prop}\textup{\cite[Proposition~3.1 \& Corollary~3.3]{gilbert2005}}
\label{prop:sz}
If $G$ is an inductive groupoid, then $\szgg$ is an inductive
groupoid with $(U,u)^{-1}=(u^{-1}U,u^{-1})$,
\begin{equation}
\label{eq:dri} \bd\left((U,u)\right)=(U,\bd(u)),\qquad
\br\left((U,u)\right)=(u^{-1}U,\br(u)),
\end{equation}
and ordering
\begin{equation*}
(U,u)\leq (V,v)\Longleftrightarrow u\leq v\text{ in }G\text{ and
}\bd(u)|V\subseteq U,
\end{equation*}
where $\bd(u)|V=\{\bd(u)|w:w\in V\}$.  For $(E,e)\in\szgg_o$ with
$(E,e)\leq\bd((A,a))$, the restriction\footnote{We need not give the corestriction explicitly, thanks to the comments following Lemma~\ref{lem:rest-ord}.} is given by
\begin{equation*}
\label{eq:rest} (E,e)|(A,a)=(E,e|u),
\end{equation*}
whilst, for any $(E,e),(F,f)\in\szgg_o$, we have
\begin{equation}
\label{eq:meet} (E,e)\wedge (F,f)=\left((e\wedge f)|(E\cup
F),e\wedge f\right).
\end{equation}
\end{prop}

The pseudoproduct in $\szgg$ may be written:
\begin{equation}
\label{eq:pseudo-sz}
(U,u)\otimes (V,v)=\left(\left.\bd(u|\br(u)\wedge\bd(v))\right|U\cup u\otimes V,u\otimes v\right),
\end{equation}
where $u\otimes V=\{u\otimes w:w\in V\}$ \cite[Theorem~3.2]{gilbert2005}.  Note also that we can inject any inductive groupoid $G$ into $\szgg$ via the mapping $\iota:G\rightarrow\szgg$, given by $g\iota=(\{\bd(g),g\},g)$.  The Szendrei expansion of an inductive groupoid will be used to prove Lemma~\ref{lem:icp5}.

\section{$(\vee,r)$-premorphisms}
\label{sec:join}

The goal of this section is the proof of the following theorem,
and its connection with the `$(\vee,i)$-premorphisms' part of
Theorem~\ref{esn}.

\begin{thm}
\label{thm:wEa-join} The category of restriction semigroups and
$(\vee,r)$-premorph\-isms is isomorphic to the category of
inductive categories and ordered functors.
\end{thm}

We note that the `objects' part of this theorem has been taken care of by the results of Section~\ref{sec:prelims}; it remains to deal with the `arrows' part.  We begin by defining the notion of a \emph{$(\vee,r)$-premorphism}.  

\begin{defn}
Let $S$ and $T$ be restriction semigroups.  A
\emph{$(\vee,r)$-premorph\-ism} is a function $\theta:S\rightarrow
T$ such that
\begin{enumerate}
    \item[($\vee 1$)] $(st)\theta\leq (s\theta)(t\theta)$;
    \item[($\vee 2$)] $s^+\theta\leq (s\theta)^+$ and $s^*\theta\leq (s\theta)^*$.
\end{enumerate}
\end{defn}

The notion of a $(\vee,r)$-premorphism generalises that of a
$(\vee,i)$-premorphism.  To see this, we must first record the
following concerning $(\vee,i)$-premorphisms:

\begin{lem}\cite[Theorem~3.1.5]{lawson1998}
\label{lem:veei}
Let $\theta:S\rightarrow T$ be a
$(\vee,i)$-premorphism.  Then $\theta$ respects inverses and the
natural partial order.
\end{lem}

We now have:

\begin{lem}
\label{lem:veei-veew} 
Let $\theta:S\rightarrow T$ be a function
between inverse semigroups.  Then $\theta$ is a
$(\vee,i)$-premorphism if, and only if, it is a
$(\vee,r)$-premorphism.
\end{lem}

\begin{proof}
($\Leftarrow$)  Immediate.

($\Rightarrow$) Suppose that $\theta:S\rightarrow T$ is a
$(\vee,i)$-premorphism.  Then
$$s^+\theta=(ss^{-1})\theta\leq (s\theta)(s^{-1}\theta)=(s\theta)(s\theta)^{-1}=(s\theta)^+,$$
using Lemma~\ref{lem:veei}.  Similarly, $s^*\theta\leq
(s\theta)^*$.
\end{proof}

We note some useful properties of $(\vee,r)$-premorphisms:

\begin{lem}
\label{lem:vee-props} Let $S$ and $T$ be restriction semigroups
with respect to semilattices $E$ and $F$, respectively.  If
$\theta:S\rightarrow T$ is a $(\vee,r)$-premorphism, then
\begin{enumerate}
    \item[(a)] $e\in E(S)\Rightarrow e\theta\in E(T)$;
    \item[(b)] $e\in E\Rightarrow e\theta\in F$;
    \item[(c)] $(s\theta)^+=s^+\theta$ and $(s\theta)^*=s^*\theta$;
    \item[(d)] $\theta$ is order-preserving.
\end{enumerate}
\end{lem}

\begin{proof}
(a) Let $e\in E(S)$.  From $(\vee 1)$, $e\theta=e^2\theta\leq
(e\theta)^2$.  Then, by \eqref{eq:npo2},
$e\theta=(e\theta)^+(e\theta)^2=(e\theta)^2$, hence $e\theta\in
E(T)$.

(b) Let $e\in E$.  From $(\vee 2)$, $e\theta=e^+\theta\leq
(e\theta)^+$.  Again by \eqref{eq:npo2},
$e\theta=(e\theta)^+(e\theta)^+=(e\theta)^+$, hence $e\theta\in
F$.

(c) We deal with the `$^+$' part; the `$^*$' part is similar.  We
will show that $(s\theta)^+\leq s^+\theta$.  The desired result
will then follow by combining this with $(\vee 2)$.  Let $s\in S$.
From $(\vee 1)$, $s\theta=(s^+s)\theta\leq (s^+\theta)(s\theta)$,
so
\begin{equation}
\label{eq:c}
s\theta=(s^+\theta)(s\theta)(s\theta)^*=(s^+\theta)(s\theta),
\end{equation}
once again by \eqref{eq:npo2}.  Applying $^+$ to both sides of
\eqref{eq:c} gives
$$(s\theta)^+=((s^+\theta)(s\theta))^+=((s^+\theta)(s\theta)^+)^+=(s^+\theta)(s\theta)^+,$$
using Lemma~\ref{lem:+}, and since $s^+\theta\in F$, by (b).
Hence $(s\theta)^+\leq s^+\theta$.

(d) Suppose that $s\leq t$ in $S$.  Then $s=s^+t$, by
\eqref{eq:npo2}, so $s\theta=(s^+t)\theta\leq
(s^+\theta)(t\theta)\leq t\theta$ in $T$.
\end{proof}

Using Lemma~\ref{lem:vee-props}(d), it is easily verified that the
composition of two $(\vee,r)$-pre\-morph\-isms is a
$(\vee,r)$-premorphism, hence restriction semigroups and
$(\vee,r)$-premorphisms constitute a category.  We are now ready
to prove the `arrows' part of Theorem~\ref{thm:wEa-join}, which we
will break down into two parts (Propositions~\ref{prop:vpre-fun} and~\ref{prop:fun-vpre}). Before we do so, however, we first
record the following, which will be used a number of times in the
remainder of this paper:

\begin{lem}
\label{lem:ordering} Let $\alpha:S\rightarrow T$ be an
order-preserving function of restriction semigroups.  We
define $\bc(\alpha):\bc(S)\rightarrow\bc(T)$ to be the same
function on the underlying sets.  Then $\bc(\alpha)$ is
order-preserving.

Let $\beta:C\rightarrow D$ be an order-preserving function of
inductive categories.  We define
$\bs(\beta):\bs(C)\rightarrow\bs(D)$ to be the same function on
the underlying sets.  Then $\bs(\beta)$ is order-preserving.
\end{lem}

\begin{proof}
Let $s,t\in \bc(S)$.  Then
\begin{align*}
s\leq t\textup{ in }\bc(S)  &\Rightarrow s\leq t\textup{ in }S
                                                         \Rightarrow s\alpha\leq t\alpha\textup{ in }T
                                                         \Rightarrow s\bc(\alpha)\leq t\bc(\alpha)\textup{ in }\bc(T).
\end{align*}
The proof of the second part is similar.
\end{proof}

We now also give a formal definition for the ordered functors
which appear in Theorem~\ref{thm:wEa-join}:

\begin{defn}
Let $\phi:C\rightarrow D$ be a function between ordered categories
$C$ and $D$.  We call $\phi$ an \emph{ordered functor} if
\begin{enumerate}
    \item $\exists x\cdot y$ in $C\Rightarrow \exists (x\phi)\cdot
    (y\phi)$ in $D$ and $(x\phi)\cdot (y\phi)=(x\cdot y)\phi$;
    \item $x\leq y$ in $C\Rightarrow x\phi\leq y\phi$ in $D$.
\end{enumerate}
\end{defn}

\begin{prop}
\label{prop:vpre-fun}
Let $S$ and $T$ be restriction semigroups with respect to
semilattices $E$ and $F$, respectively.  Let $\theta:S\rightarrow
T$ be a $(\vee,r)$-premorphism.  We define
$\thec:=\bc(\theta):\bc(S)\rightarrow\bc(T)$ to be the same
function on the underlying sets.  Then $\thec$ is an ordered
functor with respect to the restricted products in $\bc(S)$ and
$\bc(T)$.
\end{prop}
\begin{proof}
Since $\theta$ maps $E$ into $F$, plus $\bc(S)_o=E$ and
$\bc(T)_o=F$, we see that $\thec$ maps identities in $\bc(S)$ to
identities in $\bc(T)$.

We now show that $\thec$ respects restricted products.  Suppose
that $\exists s\cdot t$ in $\bc(S)$.  Then $s^*=t^+$ in $S$, hence
$(s\thec)^*=(s\theta)^*=s^*\theta=t^+\theta=(t\theta)^+=(t\thec)^+$,
by Lemma~\ref{lem:vee-props}(c).  We conclude that $\exists
(s\thec)\cdot(t\thec)$ in $\bc(T)$.  We will now use
Lemma~\ref{lem:so-useful} to show that $(s\cdot
t)\thec=(s\thec)\cdot (t\thec)$.  Note first of all that $(s\cdot t)\thec$ $=(st)\theta\leq (s\theta)(t\theta)$ and that $(s\thec)\cdot
(t\thec)=(s\theta)(t\theta)\leq (s\theta)(t\theta)$.  Now, using
Lemma~\ref{lem:+},
$$\br((s\cdot t)\thec)=(st)\theta^*=(st)^*\theta=(s^*t)^*\theta=t^*\theta,$$
since $s^*=t^+$.  Also,
\begin{align*}
\br((s\thec)\cdot (t\thec))
&=((s\theta)(t\theta))^*=((s\theta)^*(t\theta))^*=((s^*\theta)(t\theta))^*=((t^+\theta)(t\theta))^* \\
&=((t\theta)^+(t\theta))^*=(t\theta)^*=t^*\theta=\br((s\cdot
t)\thec).
\end{align*}
Hence, by Lemma~\ref{lem:so-useful}, $(s\cdot
t)\thec=(s\thec)\cdot (t\thec)$.

Finally, $\thec$ is order-preserving, by Lemma~\ref{lem:ordering}.
\end{proof}

\begin{prop}
\label{prop:fun-vpre}
Let $\phi:C\rightarrow D$ be an ordered functor of inductive
categories.  We define $\phis:=\bs(\phi):\bs(C)\rightarrow\bs(D)$
to be the same function on the underlying sets.  Then $\phis$ is a
$(\vee,r)$-premorphism with respect to the pseudoproducts in
$\bs(C)$ and $\bs(D)$.
\end{prop}
\begin{proof}
Let $s,t\in \bs(C)$.  Note that if $e=s^*\otimes t^+$, then
\begin{align*}
(s\otimes e)^*  &=(s\otimes s^*\otimes t^+)^*=(s\otimes t^+)^*=(s^*\otimes t^+)^*=s^*\otimes t^+ \\
                                &=(s^*\otimes t^+)^+=(s^*\otimes t)^+=(s^*\otimes t^+\otimes t)^+=(e\otimes t)^+,
\end{align*}
so $\exists (s\otimes e)\cdot (e\otimes t)$ and $(s\otimes e)\cdot
(e\otimes t)=s\otimes t$.  Since $\phi$ is a functor, we have
$(s\otimes t)\phi=(s\otimes e)\phi\cdot (e\otimes t)\phi=(s\otimes
e)\phi\otimes (e\otimes t)\phi,$ using the fact that $\cdot$ and
$\otimes$ coincide whenever $\cdot$ is defined. Now, $s\otimes
e\leq s$ and $e\otimes t\leq t$, so $(s\otimes e)\phi\leq s\phi$
and $(e\otimes t)\phi\leq t\phi$.  Hence
$$(s\otimes t)\phis=(s\otimes t)\phi=(s\otimes e)\phi\otimes (e\otimes t)\phi\leq (s\phi)\otimes (t\phi)=(s\phis)\otimes (t\phis).$$

Let $s\in S$.  Since functors preserve domains, we have
$(s\phis)^+=\bd(s\phi)=\bd(s)\phi=s^+\phis$. Then, in particular,
$(s\phis)^+\geq s^+\phis$.  Similarly, $(s\phis)^*\geq s^*\phis$.
\end{proof}

It is clear that if $\theta:S\rightarrow T$ is a
$(\vee,r)$-premorphism and $\phi:C\rightarrow D$ is an ordered
functor, then $\bs(\bc(\theta))=\theta$ and $\bc(\bs(\phi))=\phi$.
Furthermore, if $\theta':T\rightarrow T'$ is another
$(\vee,r)$-premorphism of restriction semigroups, and
$\phi':D\rightarrow D'$ is another ordered functor of inductive
categories, then $\bc(\theta\theta')=\bc(\theta)\bc(\theta')$ and
$\bs(\phi\phi')=\bs(\phi)\bs(\phi')$. We have therefore proved
Theorem~\ref{thm:wEa-join}.

The constructions of
Theorems~\ref{thm:wEa-ic} and~\ref{thm:ic->wEas} carry over to the
inverse case in such a way that if $S$ is an inverse semigroup and
$G$ is an inductive groupoid, then ${\bf C}(S)$ is an inductive
groupoid and ${\bf S}(G)$ is an inverse semigroup (see \cite{lawson1998}).  Then, using
Lemma~\ref{lem:veei-veew}, we see that the `$\vee$-premorph\-isms'
part of Theorem~\ref{esn} follows from Theorem~\ref{thm:wEa-join}
as a corollary.

\section{$(\wedge,r)$-premorphisms}
\label{sec:meet}

We now turn our attention to the derivation of an ESN-type Theorem
for `$\wedge$-premorphisms'.  Two results of this type will be proved in
this section.  The first will employ the more naive version of a `$(\wedge,r)$-premorphism', alluded to in the Introduction.  However, we will not be able to prove an analogue of Lemma~\ref{lem:veei-veew} for these $(\wedge,r)$-premorphisms.  In order to make the desired connection with the functions of Definition~\ref{defn:wedge-i}, and thereby deduce an ESN-type theorem for inverse semigroups and $(\wedge,i)$-premorphisms, we need the notion of a \emph{strong} $(\wedge,r)$-premorphism.  In the second part of this section, we will prove an ESN-type theorem for strong $(\wedge,r)$-premorphisms which will follow from the weaker version as a corollary.

\subsection{The weaker case: ordered $\boldsymbol{(\wedge,r)}$-premorphisms}

The goal of this subsection is the proof of the following theorem:

\begin{thm}
\label{thm:wEajoin-indcat} The category of restriction semigroups
and ordered $(\wedge,r)$-pre\-morphisms is isomorphic to the
category of inductive categories and inductive category
prefunctors.
\end{thm}

We note that the `objects' part of this theorem has been taken care of by the results of Section~\ref{sec:prelims}; it remains to deal with the `arrows' part.

We make the following definition, based upon the one-sided case in \cite{gh}:

\begin{defn}
\label{defn:meet-w} A function $\theta:S\rightarrow T$ between
restriction semigroups is called a \emph{$(\wedge,r)$-premorphism}
if
\begin{enumerate}
    \item[($\wedge 1$)] $(s\theta)(t\theta)\leq (st)\theta$;
    \item[($\wedge 2$)] $(s\theta)^+\leq s^+\theta$ and $(s\theta)^*\leq s^*\theta$.
\end{enumerate}
If, in addition, $\theta$ is order-preserving, we call it an
\emph{ordered $(\wedge,r)$-premorphism}.
\end{defn}

We note that whilst a $(\vee,r)$-premorphism is automatically
order-preserving, this is not the case for a
$(\wedge,r)$-premorphism: we must demand this explicitly.  As
noted in the Introduction, arbitrary $(\wedge,r)$-premorphisms do
not compose to give a $(\wedge,r)$-premorphism.  However, it is
easily verified that the composition of two \emph{ordered}
$(\wedge,r)$-premorphisms is an ordered $(\wedge,r)$-premorphism.
Restriction semigroups together with ordered
$(\wedge,r)$-premorphisms therefore form a category.

\begin{lem}
\label{lem:meet-props} Let $S$ and $T$ be restriction semigroups
with respect to semilattices $E$ and $F$, respectively.  Let
$\theta:S\rightarrow T$ be an ordered $(\wedge,r)$-premorphism.
If $e\in E$, then $e\theta\in F$.
\end{lem}

\begin{proof}
Let $e\in E$.  Then, by ($\wedge 2$), $(e\theta)^+\leq
e^+\theta=e\theta$.  By \eqref{eq:npo2}, we have
$(e\theta)^+=(e\theta)^+(e\theta)=e\theta\in F$.
\end{proof}

We must now define a corresponding function between inductive
categories:

\begin{defn}
\label{defn:ord-cat-pre} A function $\psi:C\rightarrow D$ between
inductive categories is called an \emph{inductive category
prefunctor} if
\begin{enumerate}
    \item[(ICP1)] $\exists s\cdot t$ in $C\Rightarrow (s\psi)\otimes (t\psi)\leq (s\cdot t)\psi$;
    \item[(ICP2)] $\bd(s\psi)\leq\bd(s)\psi$ and $\br(s\psi)\leq\br(s)\psi$;
    \item[(ICP3)] $s\leq t$ in $C\Rightarrow s\psi\leq t\psi$ in $D$;
    \item[(ICP4)] \begin{enumerate}
                                    \item[(a)] for $a\in C$ and $f\in C_o$, $a\psi|\br(a\psi)\wedge f\psi\leq (a|\br(a)\wedge f)\psi$;
                                    \item[(b)] for $a\in C$ and $e\in C_o$, $e\psi\wedge\bd(a\psi)|a\psi\leq (e\wedge\bd(a)|a)\psi$.
                                \end{enumerate}
\end{enumerate}
\end{defn}

Note that we can use Lemma~\ref{lem:apseudoe} to rewrite condition
(ICP4) in a more compact form:
\begin{enumerate}
    \item[(ICP4)$'$]    \begin{enumerate}
                                            \item[(a)] for $a\in C$ and $f\in C_o$, $a\psi\otimes f\psi\leq (a\otimes f)\psi$;
                                            \item[(b)] for $a\in C$ and $e\in C_o$, $e\psi\otimes a\psi\leq (e\otimes a)\psi$.
                                        \end{enumerate}
\end{enumerate}

\begin{lem}
Let $\psi:C\rightarrow D$ be an inductive category prefunctor of
inductive categories.  If $e\in C_o$, then $e\psi\in D_o$.
\end{lem}

\begin{proof}
Let $e\in C_o$, so that $e=e^+$ in $\bs(C)$.  By (ICP2),
$\bd(e\psi)\leq\bd(e)\psi$ in $D$, so $(e\psi)^+\leq
e^+\psi=e\psi$ in $\bs(D)$.  Then, by definition of ordering in
$\bs(D)$, $(e\psi)^+=(e\psi)^+\otimes (e\psi)=e\psi$.  Thus
$\bd(e\psi)=e\psi$ in $D$.  Hence $e\psi\in D_o$.
\end{proof}

\begin{lem}
\label{lem:comp-icp} The composition of two inductive category
prefunctors is an inductive category prefunctor.
\end{lem}

\begin{proof}
Let $\psi_1:U\rightarrow V$ and $\psi_2:V\rightarrow W$ be
inductive category prefunctors of inductive categories $U$, $V$
and $W$.  It is easy to show that $\psi_1\psi_2$ satifies
(ICP2)--(ICP4); (ICP1), however, is a little trickier.  Let $s,t\in
U$ and suppose that $\exists s\cdot t$, i.e., $\br(s)=\bd(t)$.
Put $x=s\psi_1$ and $y=t\psi_1$.  Then
\begin{align}
(s\psi_1\psi_2)\otimes (t\psi_1\psi_2)
&=(x\psi_2)\otimes (y\psi_2) \label{first-line} \\
&=(x\psi_2|\br(x\psi_2)\wedge\bd(y\psi_2))\cdot (\br(x\psi_2)\wedge\bd(y\psi_2)|y\psi_2) \notag \\
&=(x\psi_2|\br(x\psi_2)\wedge\bd(y\psi_2))\otimes (\br(x\psi_2)\wedge\bd(y\psi_2)|y\psi_2) \notag \\
\intertext{(since $\cdot$ and $\otimes$ coincide whenever $\cdot$ is defined)}
&\leq (x\psi_2|\br(x\psi_2)\wedge\bd(y)\psi_2)\otimes (\br(x)\psi_2\wedge\bd(y\psi_2)|y\psi_2) \notag \\
\intertext{(by Lemma~\ref{lem:rest-ord}, since
$\br(x\psi_2)\wedge\bd(y\psi_2)\leq\br(x\psi_2)\wedge\bd(y)\psi_2$,
etc.)}
&\leq (x|\br(x)\wedge\bd(y))\psi_2\otimes (\br(x)\wedge\bd(y)|y)\psi_2,\ \textup{by ICP4} \notag \\
&\leq ((x|\br(x)\wedge\bd(y))\cdot (\br(x)\wedge\bd(y)|y))\psi_2,\ \textup{by ICP1} \notag \\
&=(x\otimes y)\psi_2. \label{last-line}
\end{align}
But $x\otimes y=(s\psi_1)\otimes (t\psi_1)\leq (s\cdot t)\psi_1$,
by (ICP1).  Then, by (ICP3) for $\psi_2$, $(s\psi_1\psi_2)\otimes
(t\psi_1\psi_2)\leq (s\cdot t)\psi_1\psi_2$.
\end{proof}

Thus inductive categories and inductive category prefunctors form
a category.

We now prove that the functions of Definitions~\ref{defn:meet-w}
and~\ref{defn:ord-cat-pre} are indeed connected in the desired
way:

\begin{prop}
\label{prop:meetw-ipc} Let $\theta:S\rightarrow T$ be an ordered
$(\wedge,r)$-premorphism of restriction semigroups $S$ and $T$.
We define $\Theta:=\bc(\theta):\bc(S)\rightarrow\bc(T)$ to be the
same function on the underlying sets.  Then $\Theta$ is an
inductive category prefunctor with respect to the restricted
products in $\bc(S)$ and $\bc(T)$.
\end{prop}

\begin{proof}
(ICP1) Suppose that $\exists s\cdot t$.  Then
\begin{align*}
(s\Theta)\otimes (t\Theta)  &=(s\theta)(t\theta)\ \textup{(the product in $T$)} \\
                                                        &\leq (st)\theta=(s\cdot t)\theta,\ \textup{since }\exists s\cdot t \\
                                                        &=(s\cdot t)\Theta.
\end{align*}

(ICP2) We have $\bd(s\Theta)=(s\theta)^+\leq
s^+\theta=\bd(s)\Theta$.  Similarly, $\br(s\Theta)\leq
\br(s)\Theta$.

(ICP3) This follows from Lemma~\ref{lem:ordering}.

(ICP4) Let $a\in\bc(S)$ and $f\in\bc(S)_o=E$.  Then
$$a\Theta|\br(a\Theta)\wedge f\Theta=(a\theta)(a\theta)^*(f\theta)=(a\theta)(f\theta)\leq (af)\theta=(aa^*f)\theta=(a|\br(a)\wedge f)\Theta.$$
Similarly, $e\Theta\wedge \bd(a\Theta)|a\Theta\leq
(e\wedge\bd(a)|a)\Theta$, for $e\in E$.
\end{proof}

\begin{prop}
\label{prop:ocp-join} Let $\psi:C\rightarrow D$ be an inductive
category prefunctor.  We define
$\Psi:=\bs(\psi):\bs(C)\rightarrow\bs(D)$ to be the same function
on the underlying sets.  Then $\Psi$ is an ordered
$(\wedge,r)$-premorphism with respect to the pseudoproducts in
$\bs(C)$ and $\bs(D)$.
\end{prop}

\begin{proof}
($\wedge 1$) Let $s,t\in\bs(C)$.  Then $(s\Psi)\otimes
(t\Psi)=(s\psi)\otimes (t\psi)\leq (s\otimes t)\psi=(s\otimes
t)\Psi$, by an argument identical to that found between lines
\eqref{first-line} and \eqref{last-line} in
Lemma~\ref{lem:comp-icp}.

($\wedge 2$) We have:
$(s\Psi)^+=\bd(s\psi)\leq\bd(s)\psi=s^+\Psi$.  Similarly,
$(s\Psi)^*\leq s^*\Psi$.

Finally, it follows from Lemma~\ref{lem:ordering} that $\Psi$ is
order-preserving.
\end{proof}

Once again, it is easy to see that if $\theta:S\rightarrow T$ is
an ordered $(\wedge,r)$-premorphism of restriction semigroups and
$\psi:C\rightarrow D$ is an inductive category prefunctor, then
$\bs(\bc(\theta))=\theta$ and $\bc(\bs(\psi))=\psi$. Furthermore,
if $\theta':T\rightarrow T'$ is another ordered
$(\wedge,r)$-premorphism of restriction semigroups, and
$\psi':D\rightarrow D'$ is another inductive category prefunctor
of inductive categories, then
$\bc(\theta\theta')=\bc(\theta)\bc(\theta')$ and
$\bs(\phi\phi')=\bs(\phi)\bs(\phi')$.  We have therefore proved
Theorem~\ref{thm:wEajoin-indcat}.

\subsection{Strong $\boldsymbol{(\wedge,r)}$-premorphisms}

As we noted at the beginning of the section, in order to make the connection with the inverse case, we must now
introduce the intermediate step of \emph{strong}
$(\wedge,r)$-premorphisms between restriction semigroups.  Our
goal for the remainder of this section is the proof of the following corollary to Theorem~\ref{thm:wEajoin-indcat}:

\begin{thm}
\label{thm:wEajoin-indcat-strong} The category of restriction
semigroups and strong $(\wedge,r)$-pre\-morphisms is isomorphic to
the category of inductive categories and strong inductive category
prefunctors.
\end{thm}

Once again, it only remains to deal with the `arrows' part.

In the study of partial actions of one-sided restriction
semigroups \cite{gh}, it is the notion of a `strong premorphism'
which has proved most useful.  
We make the following definition in the two-sided case:

\begin{defn}
\label{defn:strongmeetpre}
Let $S$ and $T$ be restriction semigroups.  A
$(\wedge,r)$-pre\-morphism $\theta:S\rightarrow T$ is called
\emph{strong} if
\begin{enumerate}
    \item[$(\wedge 1)'$] $(s\theta)(t\theta)=(s\theta)^+(st)\theta=(st)\theta(t\theta)^*$.
\end{enumerate}
\end{defn}

It is clear that condition $(\wedge 1)$ follows from condition
$(\wedge 1)'$, by \eqref{eq:npo}.  Furthermore, we deduce the
following from \cite[Lemma~2.10(4)]{gh}:

\begin{lem}
A strong $(\wedge,r)$-premorph\-ism $\theta:S\rightarrow T$
between restriction semigroups is order-preserving.
\end{lem}

We therefore drop all explicit mention of order-preservation from
here on.

\begin{lem}
\label{lem:comp-strong} The composition of two strong
$(\wedge,r)$-premorphisms is a strong $(\wedge,r)$-premorphism.
\end{lem}

\begin{proof}
Let $\theta_1:U\rightarrow V$ and $\theta_2:V\rightarrow W$ be
strong $(\wedge,r)$-premorphisms of restriction semigroups $U$,
$V$ and $W$. Condition $(\wedge 2)$ is immediate:
$(s\theta_1\theta_2)^+\leq ((s\theta_1)^+)\theta_2\leq
s^+\theta_1\theta_2.$ Similarly, $(s\theta_1\theta_2)^*\leq
s^*\theta_1\theta_2$.

For $(\wedge 1)'$, we have
$(s\theta_1\theta_2)(t\theta_1\theta_2)=(s\theta_1\theta_2)^+((s\theta_1)(t\theta_1))\theta_2$,
using $(\wedge 1)'$ for $\theta_2$.  Now,
\begin{align*}
(s\theta_1\theta_2)^+\leq ((s\theta_1)^+)\theta_2
&\Longrightarrow (s\theta_1\theta_2)^+\leq ((s\theta_1)^+)\theta_2^+ \\
&\Longrightarrow
(s\theta_1\theta_2)^+=(s\theta_1\theta_2)^+((s\theta_1)^+)\theta_2^+,
\end{align*}
by \eqref{eq:npo2}, so
\begin{align*}
(s\theta_1\theta_2)(t\theta_1\theta_2)
&=(s\theta_1\theta_2)^+((s\theta_1)^+)\theta_2^+((s\theta_1)(t\theta_1))\theta_2 \\
&=(s\theta_1\theta_2)^+((s\theta_1)^+)\theta_2^+((s\theta_1)^+(st)\theta_1)\theta_2 \\
\intertext{(using $(\wedge 1)'$ for $\theta_1$)}
&=(s\theta_1\theta_2)^+((s\theta_1)^+)\theta_2(st)\theta_1\theta_2 \\
\intertext{(applying
$(a\theta_2)(b\theta_2)=(a\theta_2)^+(ab)\theta_2$ with
$a=(s\theta_1)^+$ and $b=(st)\theta_1$)}
&=(s\theta_1\theta_2)^+(st)\theta_1\theta_2,
\end{align*}
since $(s\theta_1\theta_2)^+\leq ((s\theta_1)^+)\theta_2$.
Similarly,
$(s\theta_1\theta_2)(t\theta_1\theta_2)=(st)\theta_1\theta_2(s\theta_1\theta_2)^*$.
\end{proof}

Thus restriction semigroups together with strong
$(\wedge,r)$-pre\-morph\-isms form a category.

We now define a function between inductive categories which will
correspond to a strong $(\wedge,r)$-premorphism between
restriction semigroups.

\begin{defn}
\label{defn:sicp} An inductive category prefunctor
$\psi:C\rightarrow D$ will be called \emph{strong} if
\begin{enumerate}
\item[(ICP5)]   \begin{enumerate}
                                        \item[(a)] $\bd((s\psi)\otimes (t\psi))=\bd(s\psi)\wedge\bd((s\otimes t)\psi)$;
                                        \item[(b)] $\br((s\psi)\otimes (t\psi))=\br((s\otimes t)\psi)\wedge\br(t\psi)$.
                                \end{enumerate}
\end{enumerate}
\end{defn}

We note the following:

\begin{lem}
Condition (ICP4) follows from (ICP5).
\end{lem}

\begin{proof}
Let $\psi:C\rightarrow D$ be a function between inductive categories which satisfies conditions (ICP1)--(ICP3), plus (ICP5).  We consider (ICP4)$'$(a) and observe that, for $a\in C$ and $f\in C_o$,
\begin{align}
a\psi\otimes f\psi\leq (a\otimes f)\psi\textup{ in }D	&\Longleftrightarrow a\psi\otimes f\psi\leq (a\otimes f)\psi\textup{ in }{\bf S}(D) \notag \\
																											&\Longleftrightarrow a\psi\otimes f\psi=(a\otimes f)\psi\otimes\br(a\psi\otimes f\psi) \notag \\
																											&\Longleftrightarrow a\psi\otimes f\psi=(a\otimes f)\psi\otimes\left(\br(a\psi)\wedge f\psi\right), \label{eq:latter}
\end{align}
using Lemma~\ref{lem:apseudoe}.  We will use Lemma~\ref{lem:so-useful}, in conjunction with (ICP5), to demonstrate the equality \eqref{eq:latter}.  We note first of all that $a\psi\otimes f\psi\leq a\psi$ and also that
$$(a\otimes f)\psi\otimes\left(\br(a\psi)\wedge f\psi\right)\leq (a\otimes f)\psi\leq a\psi,$$
since $a\otimes f\leq a$ and $\psi$ is order-preserving.  It remains to show that each side of \eqref{eq:latter} has the same range.  On the one hand, we have
$$\br(a\psi\otimes f\psi)=\br\left((a\otimes f)\psi\right)\wedge\br(f\psi)=\br\left((a\otimes f)\psi\right)\wedge f\psi,$$
by (ICP5)(b).  On the other, using Lemma~\ref{lem:apseudoe},
\begin{align*}
(a\otimes f)\psi\otimes\left(\br(a\psi)\wedge f\psi\right)	&=\left.(a\otimes f)\psi\right|\br\left((a\otimes f)\psi\right)\wedge\br(a\psi)\wedge f\psi, \\
																														&=\left.(a\otimes f)\psi\right|\br\left((a\otimes f)\psi\right)\wedge f\psi,
\end{align*}
since $\br\left((a\otimes f)\psi\right)\leq\br(a\psi)$.  Thus
$$\br\left((a\otimes f)\psi\otimes\left(\br(a\psi)\wedge f\psi\right)\right)=\br\left((a\otimes f)\psi\right)\wedge f\psi=\br(a\psi\otimes f\psi),$$
as required.  We conclude that (ICP4)$'$(a) holds.  Part (b) follows similarly.
\end{proof}

Thus a strong inductive category prefunctor may be regarded as a function defined by conditions (ICP1)--(ICP3) and (ICP5) only.

We defer the proof that the composition of two strong inductive
category prefunctors is a strong inductive category prefunctor
until after the following propositions, which prove that the functions of Definitions~\ref{defn:strongmeetpre} and~\ref{defn:sicp} are once again connected in the desired way:

\begin{prop}
\label{prop:strong-sicp} Let $\theta:S\rightarrow T$ be a strong
$(\wedge,r)$-premorphism of restriction semigroups.  We define
$\Theta:=\bc(\theta):\bc(S)\rightarrow\bc(T)$ to be the same
function on the underlying sets.  Then $\Theta$ is a strong
inductive category prefunctor with respect to the restricted
products in $\bc(S)$ and $\bc(T)$.
\end{prop}

\begin{proof}
By Proposition~\ref{prop:meetw-ipc}, $\Theta$ is an inductive
category prefunctor.

(ICP5)(a) We have
\begin{align*}
\bd((s\Theta)\otimes (t\Theta))
&=((s\theta)(t\theta))^+=((s\theta)^+(st)\theta)^+=((s\theta)^+(st)\theta^+)^+ \\
&=(s\theta)^+(st)\theta^+=\bd(s\Theta)\wedge\bd((s\otimes
t)\Theta).
\end{align*}
Part (b) is similar.
\end{proof}

\begin{prop}
\label{prop:sicp-strong} Let $\psi:C\rightarrow D$ be a strong
inductive category prefunctor.  We define
$\Psi:=\bs(\psi):\bs(C)\rightarrow\bs(D)$ to be the same function
on the underlying sets.  Then $\Psi$ is a strong
$(\wedge,r)$-premorphism with respect to the pseudoproducts in
$\bs(C)$ and $\bs(D)$.
\end{prop}

\begin{proof}
By Proposition~\ref{prop:ocp-join}, $\Psi$ is an ordered
$(\wedge,r)$-premorphism.

We know that $(s\Psi)\otimes (t\Psi)\leq (s\otimes t)\Psi$ and
that $(s\Psi)^+\otimes (s\otimes t)\Psi\leq (s\otimes t)\Psi$, by \eqref{eq:npo}.  We will use
Lemma~\ref{lem:so-useful} to show that condition $(\wedge 1)'$
holds.  It remains to show that both sides of the desired equality
have the same domain.  From (ICP5), we have:
\begin{align*}
\bd((s\Psi)\otimes (t\Psi)) &=\bd(s\Psi)\wedge\bd((s\otimes t)\Psi)=\bd(\bd(s\Psi)\wedge\bd((s\otimes t)\Psi)) \\
                                                        &=((s\Psi)^+\otimes (s\otimes t)\Psi^+)^+=((s\Psi)^+\otimes (s\otimes t)\Psi)^+ \\
                                                        &=\bd((s\Psi)^+\otimes (s\otimes t)\Psi).
\end{align*}
So by Lemma~\ref{lem:so-useful}, $(s\Psi)\otimes
(t\Psi)=(s\Psi)^+\otimes (s\otimes t)\Psi$.  By a similar
argument, $(s\Psi)\otimes (t\Psi)=(s\otimes t)\Psi\otimes
(t\Psi)^*$.
\end{proof}

It is clear that if $\theta:S\rightarrow T$ is a strong
$(\wedge,r)$-premorphism and $\psi:C\rightarrow T$ is a strong
inductive category prefunctor, then $\bs(\bc(\theta))=\theta$ and
$\bc(\bs(\psi))=\psi$.  Furthermore, if $\theta':T\rightarrow T'$
is another strong $(\wedge,r)$-premorphism, then
$\bc(\theta\theta')=\bc(\theta)\bc(\theta')$.  In order to complete the proof of Theorem~\ref{thm:wEajoin-indcat-strong}, we must prove that the composition of two strong inductive category prefunctors is also a strong inductive category prefunctor, thereby showing that inductive categories together with strong inductive category prefunctors do indeed form a category.

\begin{lem}
\label{lem:strong-comp}
The composition of two strong inductive category prefunctors is a
strong inductive category prefunctor.
\end{lem}

\begin{proof}
We take an indirect approach using the preceding proposition.  Let
$\psi_1:U\rightarrow V$ and $\psi_2:V\rightarrow W$ be strong
inductive category prefunctors between inductive categories $U$,
$V$ and $W$.  By Proposition~\ref{prop:sicp-strong}, we can
construct strong $(\wedge,r)$-premorphisms
$\bs(\psi_1):\bs(U)\rightarrow\bs(V)$ and
$\bs(\psi_2):\bs(V)\rightarrow\bs(W)$.  Then
$\bs(\psi_1)\bs(\psi_2):\bs(U)\rightarrow\bs(W)$ is a strong
$(\wedge,r)$-premorphism, by Lemma~\ref{lem:comp-strong}.  Now,
$\psi_1\psi_2$ is certainly an inductive category prefunctor (but
is not necessarily strong), in which case, $\bs(\psi_1\psi_2)$ is
an ordered $(\wedge,r)$-premorphism, by
Proposition~\ref{prop:ocp-join}.  Let `$\sim$' denote the
relationship `...is the same function on the underlying sets
as...'.  Then $\psi_1\sim\bs(\psi_1)$ and $\psi_2\sim\bs(\psi_2)$,
so $\psi_1\psi_2\sim\bs(\psi_1)\bs(\psi_2)$.  But
$\psi_1\psi_2\sim\bs(\psi_1\psi_2)$.  We deduce that
$\bs(\psi_1)\bs(\psi_2)=\bs(\psi_1\psi_2)$, hence
$\bs(\psi_1\psi_2)$ is a strong $(\wedge,r)$-premorphism.  Then
$\bc(\bs(\psi_1\psi_2))=\psi_1\psi_2$ is a strong inductive
category prefunctor, by Proposition~\ref{prop:strong-sicp}.
\end{proof}

Thus Theorem~\ref{thm:wEajoin-indcat-strong} is proved.

\section{The inverse case}
\label{sec:inv}

We will now deduce a corollary to
Theorem~\ref{thm:wEajoin-indcat-strong} in the inverse case.  Our
notion of `$\wedge$-premorphism' will simply be that of
Definition~\ref{defn:wedge-i}.  In addition, if such a
$(\wedge,i)$-premorphism is order-preserving, then we will call it
an \emph{ordered} $(\wedge,i)$-premorph\-ism.  We will prove the
following:

\begin{thm}
\label{thm:esn-new} 
The category of inverse semigroups and ordered
$(\wedge,i)$-premorph\-isms is isomorphic to the category of
inductive groupoids and ordered groupoid premorphisms.
\end{thm}

Note that the `objects' part of the above theorem is taken care of by the comments at the end of Section~\ref{sec:join}.

We deduce the following from \cite[Lemma~2.12]{gh}:
\begin{lem}
\label{lem:strongpre-inverse2}
Let $\theta:S\rightarrow T$ be a
mapping between inverse semigroups.  Then $\theta$ is an ordered
$(\wedge,i)$-premorphism if, and only if, it is a strong
$(\wedge,r)$-premorphism of restriction semigroups.
\end{lem}

We now, of course, need a corresponding function between inductive
groupoids.  In his study of the partial actions of inductive
groupoids, Gilbert \cite{gilbert2005} has employed the following
definition:
\begin{defn}\cite[p.~184]{gilbert2005}
\label{defn:ord-grpd-pre}
A function $\psi:G\rightarrow H$ between
inductive groupoids will be called an \emph{ordered groupoid premorphism} if the following conditions are satisfied:
\begin{enumerate}
    \item[(ICP1)] if the product $g\cdot h$ is defined in $G$, then $(g\psi)\otimes (h\psi)\leq (g\cdot h)\psi$;
    \item[(IGP)] $(g\psi)^{-1}=g^{-1}\psi$;
    \item[(ICP3)] if $g\leq h$ in $G$, then $g\psi\leq h\psi$ in $H$.
\end{enumerate}
\end{defn}

We want to show that Gilbert's ordered groupoid premorphisms are a special case of our strong inductive category prefunctors.  In particular, as a first step towards establishing an analogue of Lemma~\ref{lem:strongpre-inverse2}, we need to show that an ordered groupoid premorphism satisfies condition (ICP5).  Unfortunately, an elementary proof of this does not seem to be forthcoming; instead, we provide a proof which employs the machinery of the Szendrei expansion of an inductive groupoid, as introduced in Section~\ref{sec:sz}, together with the following result of Gilbert:

\begin{thm}\textup{\cite[Theorem 4.4 \& Proposition 4.6]{gilbert2005}}
\label{thm:grpd-exp}
Let $G$ and $H$ be inductive groupoids.  If
$\psi:G\rightarrow H$ is an ordered groupoid premorphism, then
there exists a unique inductive functor $\pb:\szgg\rightarrow H$
such that $\iota\pb=\psi$.  Conversely, if $\pb:\szgg\rightarrow
H$ is an inductive functor, then $\psi:=\iota\pb$ is an ordered groupoid premorphism.
\end{thm}

\begin{lem}
\label{lem:icp5}
Let $\psi:G\rightarrow H$ be an ordered groupoid premorphism.  Then $\psi$ satisfies condition (ICP5).
\end{lem}

\begin{proof}
By Theorem~\ref{thm:grpd-exp}, $\psi$ may be decomposed as $\psi=\iota\pb$, where $\iota:G\rightarrow\szgg$ is as in Section~\ref{sec:sz} and $\pb:\szgg\rightarrow H$ is an inductive functor.  Note that an inductive functor preserves products, domains, ranges, ordering, restrictions, corestrictions and meets.  We will demonstrate that (ICP5)(a) holds; part (b) is similar.    
Let $s,t\in G$.  Then $s\psi=s\iota\pb=(\{\bd(s),s\},s)\pb$ and $t\psi=t\iota\pb=(\{\bd(t),t\},t)\pb$.  We have
\begin{align*}
s\psi\otimes t\psi	&=(\{\bd(s),s\},s)\pb\ \otimes\ (\{\bd(t),t\},t)\pb \\
										&=\left[\,(\{\bd(s),s\},s)\ \otimes\ (\{\bd(t),t\},t)\,\right]\pb \\
										&=\left(\,\left.\Delta\right|\{\bd(s),s\}\cup s\otimes\{\bd(t),t\},\ s\otimes t\,\right)\pb,
\end{align*}
by \eqref{eq:pseudo-sz}, where $\Delta=\bd(s|\br(s)\wedge\bd(t))=\bd(s\otimes t)$.  
Now, $\Delta\leq\bd(s)$, so $\Delta|\bd(s)=\Delta$.  Also, $s\otimes\bd(t)=s|\br(s)\wedge\bd(t)=\Delta|s$, by Lemma~\ref{lem:so-useful}, so
$$s\psi\otimes t\psi=\left(\,\{\Delta,\,\Delta|s,\,s\otimes t\},\ s\otimes t\,\right)\pb.$$
Then, by \eqref{eq:dri},
$$\bd(s\psi\otimes t\psi)=\left(\,\{\Delta,\,\Delta|s,\,s\otimes t\},\ \Delta\,\right)\pb.$$

On the other hand,
\begin{align*}
(s\otimes t)\psi=(s\otimes t)\iota\pb	&=\left(\,\{\bd(s\otimes t),\,s\otimes t\},\ s\otimes t\,\right)\pb=\left(\,\{\Delta,\,s\otimes t\},\ s\otimes t\,\right)\pb,
\end{align*}
so that
$$\bd((s\otimes t)\psi)=\left(\,\{\Delta,\,s\otimes t\},\ \Delta\,\right)\pb.$$
Note that $\bd(s\psi)=(\{\bd(s),s\},\bd(s))\pb$.  We have
\begin{align*}
\bd(s\psi)\wedge\bd((s\otimes t)\psi)
&=(\,\{\bd(s),\,s\},\ \bd(s)\,)\pb\ \wedge\ \left(\,\{\Delta,\,s\otimes t\},\ \Delta\,\right)\pb \\
&=\left[\,(\,\{\bd(s),\,s\},\,s),\ \bd(s)\,)\ \wedge\ (\,\{\Delta,\,s\otimes t\},\ \Delta\,)\,\right]\pb \\
&=\left(\,\left.\Delta\right|\{\bd(s),\,s,\,\Delta,\,s\otimes t\},\ \Delta\,\right)\pb\quad \textup{(by \eqref{eq:meet})} \\
&=\left(\,\{\Delta,\,\Delta|s,\,s\otimes t\},\ \Delta\,\right)\pb \\
\intertext{(since $\Delta|(s\otimes t)=s\otimes t$, by Lemma~\ref{lem:a=dab})}
&=\bd(s\psi\otimes t\psi),
\end{align*}
as required.
\end{proof}

We note the following:

\begin{lem}\cite[Lemma~4.2]{gilbert2005}
\label{lem:ord-grpd-pre} 
Let $\psi:G\rightarrow H$ be an ordered groupoid premorphism.  Then, 
for any $g\in G$, $\bd(g\psi)\leq \bd(g)\psi$ and $\br(g\psi)\leq \br(g)\psi$.
\end{lem}

That is, an ordered groupoid premorphism satisfies condition (ICP2).  Thus, by the two preceding lemmas, every ordered groupoid premorphism is a strong inductive category prefunctor.

\begin{lem}
The composition of two ordered groupoid premorphisms is an ordered groupoid premorphism.
\end{lem}

\begin{proof}
The composition of two ordered groupoid premorphisms is certainly a strong inductive category prefunctor, by Lemma~\ref{lem:strong-comp}.  Condition (IGP) follows easily.
\end{proof}

Inductive groupoids together with ordered groupoid premorphisms therefore form a category.

\begin{prop}
\label{prop:wedgei-igp} Let $\theta:S\rightarrow T$ be an ordered
$(\wedge,i)$-premorphism.  We define
$\Theta:=\bc(\theta):\bc(S)\rightarrow\bc(T)$ to be the same
function on the underlying sets.  Then $\Theta$ is an ordered
groupoid premorphism with respect to the restricted products in
$\bc(S)$ and $\bc(T)$.
\end{prop}

\begin{proof}
By Lemma~\ref{lem:strongpre-inverse2}, $\theta$ is a strong
$(\wedge,r)$-premorphism.  Then $\Theta$ is a strong inductive
category prefunctor, by Proposition~\ref{prop:strong-sicp}.  It
remains to show that $\Theta$ satisfies (IGP):
$(g\Theta)^{-1}=(g\theta)^{-1}=g^{-1}\theta=g^{-1}\Theta.$
\end{proof}

\begin{prop}
\label{prop:igp-owip} Let $\psi:G\rightarrow H$ be an ordered
groupoid premorphism.  We define
$\Psi:=\bs(\psi):\bs(G)\rightarrow\bs(H)$ to be the same function
on the underlying sets.  Then $\Psi$ is an ordered
$(\wedge,i)$-premorphism with respect to the pseudoproducts in
$\bs(G)$ and $\bs(H)$.
\end{prop}

\begin{proof}
We know that $\psi$ is a strong inductive category prefunctor.  Then, by
Proposition~\ref{prop:sicp-strong}, $\Psi$ is a strong $(\wedge,r)$-premorphism.  It remains to show that condition
$(\wedge 2)'$ is satisfied:
$(g\Psi)^{-1}=(g\psi)^{-1}=g^{-1}\psi=g^{-1}\Psi,$ as required.
\end{proof}

Thus ordered groupoid premorphisms and ordered $(\wedge,i)$-premorphisms are connected in the manner required to prove Theorem~\ref{thm:esn-new}.  We can now complete our inductive groupoid analogue of Lemma~\ref{lem:strongpre-inverse2}:

\begin{lem}
\label{lem:igp-sicp}
Let $\psi:G\rightarrow H$ be a function between inductive
groupoids.  Then $\psi$ is an ordered groupoid premorphism if,
and only if, it is a strong inductive category prefunctor.
\end{lem}

\begin{proof}
$(\Rightarrow)$  Suppose that $\psi:G\rightarrow H$ is an
ordered groupoid premorphism.  It follows from Lemmas~\ref{lem:icp5} and~\ref{lem:ord-grpd-pre} that $\psi$ is a strong inductive category prefunctor.

$(\Leftarrow)$  Suppose that $\psi:G\rightarrow H$ is a strong
inductive category prefunctor of inductive groupoids.  By
Proposition~\ref{prop:sicp-strong}, we have a strong
$(\wedge,r)$-premorphism $\bs(\psi):\bs(G)\rightarrow\bs(H)$.  But
this is 
a strong $(\wedge,r)$-premorphism 
of inverse semigroups, so, by
Lemma~\ref{lem:strongpre-inverse2}, it is an ordered
$(\wedge,i)$-premorphism.  Then $\bc(\bs(\psi))=\psi$ is an
ordered groupoid premorphism, by
Proposition~\ref{prop:wedgei-igp}.
\end{proof}

As ever, it is clear that if $\theta:S\rightarrow T$ is an ordered
$(\wedge,i)$-premorphism and $\psi:G\rightarrow H$ is an ordered
groupoid premorphism, then $\bs(\bc(\theta))=\theta$ and
$\bc(\bs(\psi))=\psi$.  Also, if $\theta':T\rightarrow T'$ is
another ordered $(\wedge,i)$-premorphism of inverse semigroups,
and $\psi':H\rightarrow H'$ is another ordered groupoid
premorphism, then $\bc(\theta\theta')=\bc(\theta)\bc(\theta')$ and
$\bs(\psi\psi')=\bs(\psi)\bs(\psi')$.  We have proved
Theorem~\ref{thm:esn-new}.

Let ${\bf REST}$ denote a category whose objects are restriction semigroups; subscripts of `mor', `strong', `$\wedge$'
and `$\vee$' will denote that the arrows of the category are
(2,1,1)-morphisms, strong $(\wedge,r)$-premorphisms, ordered
$(\wedge,r)$-premorph\-isms and $(\vee,r)$-premorphisms, respectively.
Thus, for example, ${\bf REST}_\textup{strong}$ denotes the
category of restriction semigroups and strong
$(\wedge,r)$-premorphisms.  Similarly, ${\bf INV}$ will denote a
category whose objects are inverse semigroups; the subscripts
applied here will be `mor', `$\wedge$' and `$\vee$', denoting
morphisms, ordered $(\wedge,i)$-premorphisms and
$(\vee,i)$-premorphisms, respectively.  A category whose objects are
inductive categories or inductive groupoids will be denoted by
${\bf IC}$ or ${\bf IG}$, as appropriate.  A subscript of `ord' or
`ind' will denote ordered functors or inductive functors,
respectively, whilst `pre', `strong' and `ogp' will denote inductive
category prefunctors, strong inductive
category prefunctors and ordered groupoid premorphisms, respectively.  We can now summarise the
connections between the various categories in this paper in the
following pair of Hasse diagrams: \vspace{0.75cm}

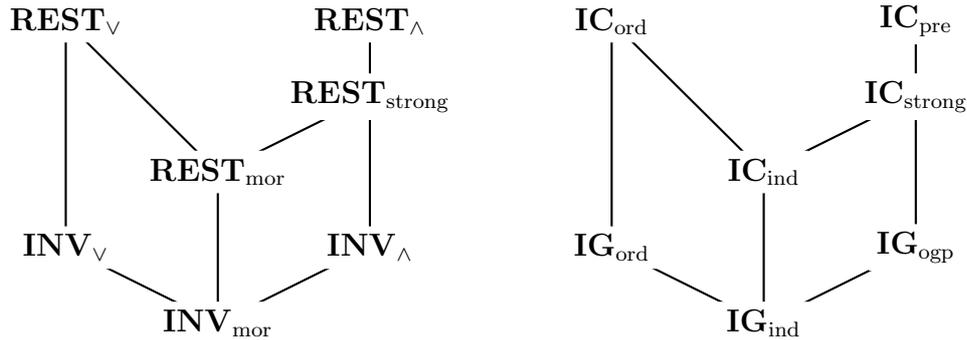
\begin{center}
\begin{tabular}{ccc}
\begin{pspicture}(-2,-2)(2,2)
\psline(-2,2)(-2,-1) \psline(-2,2)(0,0) \psline(2,2)(2,1)
\psline(2,1)(0,0) \psline(2,1)(2,-1) \psline(0,0)(0,-2)
\psline(-2,-1)(0,-2) \psline(2,-1)(0,-2)

\rput*(-2,2){${\bf REST}_\vee$}
\rput*(2,2){${\bf REST}_\wedge$}

\rput*(2,1){${\bf REST}_\textup{strong}$}
\rput*(0,0){${\bf REST}_\textup{mor}$}

\rput*(-2,-1){${\bf INV}_\vee$} \rput*(2,-1){${\bf INV}_\wedge$}

\rput*(0,-2){${\bf INV}_\textup{mor}$}
\end{pspicture}
 & \qquad\qquad\qquad\qquad &
\begin{pspicture}(-2,-2)(2,2)
\psline(-2,2)(-2,-1) \psline(-2,2)(0,0) \psline(2,2)(2,1)
\psline(2,1)(0,0) \psline(2,1)(2,-1) \psline(0,0)(0,-2)
\psline(-2,-1)(0,-2) \psline(2,-1)(0,-2)

\rput*(-2,2){${\bf IC}_\textup{ord}$} \rput*(2,2){${\bf
IC}_\textup{pre}$}

\rput*(2,1){${\bf IC}_\textup{strong}$}

\rput*(0,0){${\bf IC}_\textup{ind}$}

\rput*(-2,-1){${\bf IG}_\textup{ord}$} \rput*(2,-1){${\bf
IG}_\textup{ogp}$}

\rput*(0,-2){${\bf IG}_\textup{ind}$}
\end{pspicture}
\end{tabular}
\end{center}
\vspace{0.75cm}

Each category in the left-hand diagram is isomorphic to the
corresponding category in the right-hand diagram, and vice versa.


\begin{thebibliography}{99}
		\bibitem{armstrong1984} S.~Armstrong, The structure of type A semigroups, \textit{Semigroup Forum}, \textbf{29} (1984), 319-336.
    \bibitem{br1984} J.-C.~Birget and J.~Rhodes, Almost finite expansions of arbitrary semigroups, \textit{Journal of Pure and Applied Algebra}, \textbf{32} (1984), 239-287.
    \bibitem{cm} J.~R.~B.~Cockett and E.~Manes, Boolean and classical restriction categories, preprint, 2007.
    \bibitem{fountain1977b} J.~Fountain, A class of right \textit{PP} monoids, \textit{Quarterly Journal of Mathematics, Oxford (2)}, \textbf{28} (1977), 285-300.
    \bibitem{fountain1979} J.~Fountain, Adequate semigroups, \textit{Proceedings of the Edinburgh Mathematical Society (2)}, \textbf{22} (1979), 113-125.
    \bibitem{fg1990} J.~Fountain and G.~M.~S.~Gomes, The Szendrei expansion of a semigroup, \textit{Mathematika}, \textbf{37} (1990), 251-260.
    \bibitem{gilbert2005} N.~D.~Gilbert, Actions and expansions of ordered groupoids, \textit{Journal of Pure and Applied Algebra}, \textbf{198} (2005), 175-195.
    \bibitem{gg2003} G.~M.~S.~Gomes and V.~Gould, Finite proper covers in a class of finite semigroups with commuting idempotents, \textit{Semigroup Forum}, \textbf{66} (2003), 433-454.
    \bibitem{gould} V.~Gould, (Weakly) left $E$-ample semigroups,\newline
\texttt{http://www-users.york.ac.uk/$\sim$varg1/finitela.ps}
    \bibitem{gh} V.~Gould and C.~Hollings, Partial actions of inverse and weakly left $E$-ample semigroups, to appear in \textit{Journal of the Australian Mathematical Society}.
    \bibitem{thesis} C.~Hollings, \textit{Partial Actions of Semigroups and Monoids}, Ph.D.\ thesis, University of York, 2007.
    \bibitem{h2007} C.~Hollings, Partial actions of monoids, \textit{Semigroup Forum}, \textbf{75}(2) (2007), 293-316.
    \bibitem{js2001} M.~Jackson and T.~Stokes, An invitation to $C$-semigroups, \textit{Semigroup Forum}, \textbf{62} (2001), 279-310.
    \bibitem{jacobson1980} N.~Jacobson, \textit{Basic Algebra}, Volume II, W.~H.~Freeman and Co., San Francisco, 1980.
    \bibitem{lawson1991} M.~V.~Lawson, Semigroups and ordered categories I: the reduced case, \textit{Journal of Algebra}, \textbf{141} (1991), 422-462.
    \bibitem{lawson1998} M.~V.~Lawson, \textit{Inverse Semigroups: The Theory of Partial Symmetries}, World Scientific, 1998.
    \bibitem{mr1977} D.~B.~McAlister and N.~R.~Reilly, $E$-unitary covers for inverse semigroups, \textit{Pacific Journal of Mathematics}, \textbf{68} (1977), 161-174.
    \bibitem{szendrei1989} M.~B.~Szendrei, A note on Birget-Rhodes expansion of groups, \textit{Journal of Pure and Applied Algebra}, \textbf{58} (1989), 93-99.
\end{thebibliography}
\end{document}